\documentclass[11pt]{amsart}

\usepackage{rlepsf, 
amsthm, amsfonts}

\newcommand{\df}{\ensuremath{\partial}} 
\newcommand{\dfe}{\ensuremath{\partial^\varepsilon}} 

\newcommand{\dd}{\ensuremath{\delta}}           
\newcommand{\dde}{\ensuremath{\delta^\varepsilon}}
\newcommand{\aug}{\ensuremath{\varepsilon}}

\newcommand{\alg}{\ensuremath{\mathcal{A}}} 

\newcommand{\m}{\mathbf{m}}
\newcommand{\n}{\mathbf{n}}
\newcommand{\bmu}{\boldsymbol{\mu}}

\DeclareMathOperator{\img}{Im}
\DeclareMathOperator{\id}{Id}

\newcommand{\rr}{\ensuremath{\mathbb{R}}}
\newcommand{\zz}{\ensuremath{\mathbb{Z}}}

\theoremstyle{plain}
\newtheorem{thm}{Theorem}[section]
\newtheorem{cor}[thm]{Corollary}
\newtheorem{lem}[thm]{Lemma}

\newtheorem{prop}[thm]{Proposition}

\theoremstyle{definition}

\newtheorem{ex}[thm]{Example}

\theoremstyle{remark}
\newtheorem{rem}[thm]{Remark}

\numberwithin{equation}{section}
\def\dfn#1{{\em #1}}

\begin{document}

\title{Product Structures for Legendrian Contact Homology}

\author[G. Civan]{Gokhan Civan} \address{University of Maryland,
  College Park, MD 20742} \email{gcivan@math.umd.edu}

\author[J. Etynre]{John B. Etnyre} \address{Georgia Insitute of
  Technology, Atlanta, GA 30332} \email{etnyre@math.gatech.edu}

\author[P. Koprowski]{Paul Koprowski} \address{University of Maryland,
  College Park, MD 20742} \email{pkoprows@math.umd.edu}

\author[J. Sabloff]{Joshua M. Sabloff} \address{Haverford College,
  Haverford, PA 19041} \email{jsabloff@haverford.edu}

\author[A. Walker]{Alden Walker} \address{California Institute of
  Technology, Pasadena, CA 91125} \email{awalker@caltech.edu}

\begin{abstract}
  Legendrian contact homology (LCH) and its associated differential
  graded algebra are powerful non-classical invariants of Legendrian
  knots.  Linearization makes the LCH computationally tractable at the
  expense of discarding nonlinear (and noncommutative) information.
  To recover some of the nonlinear information while preserving
  computability, we introduce invariant cup and Massey products ---
  and, more generally, an $A_\infty$ structure --- on the linearized
  LCH.  We apply the products and $A_\infty$ structure in three ways:
  to find infinite families of Legendrian knots that are not isotopic
  to their Legendrian mirrors, to reinterpret the duality theorem of
  the fourth author in terms of the cup product, and to recover
  higher-order linearizations of the LCH.
\end{abstract}

\maketitle

\section{Introduction}
\label{sec:intro}

A central problem in the theory of Legendrian knots in the standard
contact $3$-space is to produce effective invariants and understand
their geometric meaning.  The first ``classical'' invariants of
Legendrian knots were the Thurston-Bennequin and rotation numbers
\cite{bennequin}.  These two invariants classify Legendrian knots in
the standard contact structure when the underlying smooth knot type is
the unknot \cite{yasha-fraser}, a torus knot, or the figure eight knot
\cite{etnyre-honda:knots}; see also \cite{ding-geiges:knots}.

These early results raised the question of whether non-isotopic
Legendrian knots with the same classical invariants exist.  A
particular instance of this question was Fuchs and Tabachnikov's
Legendrian mirror question \cite{f-t}: given a Legendrian knot $K$
with rotation number zero, is it isotopic to its image $\overline{K}$
under the contactomorphism $(x,y,z) \mapsto (x,-y,-z)$?  This map is
isotopic to the identity through diffeomorphisms but not
contactomorphisms (it changes the sign of the contact form).  New
invariants, beginning with Legendrian contact homology \cite{chv,
  yasha:icm} and followed by normal rulings \cite{chv-pushkar} and the
Knot Floer Homology Legendrian invariant \cite{ost:leg-trans}, have
been used to find non-isotopic Legendrian knots with the same
classical invariants.  In this paper, we study the algebraic structure
of the Legendrian contact homology differential graded algebra (DGA)
and how it can be used to define computable invariants of Legendrian
knots that are stronger that Chekanov's linearization and that detect
the geometric property of a Legendrian knot not being isotopic to its
Legendrian mirror.

The Legendrian contact homology of a Legendrian knot $K$ is the
homology of a free non-commutative DGA $(\alg_K,\df)$ over $\zz_2$
whose differential is nonlinear, so it is extremely hard to exploit
directly.  Several methods have been devised to extract useful
information from the DGA.  The most tractable of these is Chekanov's
method of linearization, which uses an ``augmentation'' $\aug: \alg_K
\to \zz_2$ to produce a finite-dimensional chain complex whose
homology is denoted $LCH^\aug_*(K)$ \cite{chv}.  The loss of
noncommutative structure, however, means that linearized homology is
unable to detect any differences between a Legendrian knot and its
mirror; this is also true of another easily computable invariant, a
normalized count of augmentations \cite{ns:augm-rulings}. Chekanov
also defined higher-order linearizations that take nonlinear parts of
the differential into account, but these have not yet proved to be any
more effective than the original (order one) linearization.  Still
another method, Ng's characteristic algebra, retains the nonlinear
structure of the DGA and can be used to distinguish a Legendrian $6_2$
knot from its mirror \cite{lenny:computable}, but its practical use is
more art than algorithm.

In this paper, we develop a new method of extracting nonlinear
information from the DGA, namely by defining cup and Massey product
structures --- and even $A_n$ and $A_\infty$ structures --- on the
linearized cohomology $LCH^*_\aug(K)$.  The cup product has already
appeared implicitly in the fourth author's investigation of duality
for the linearized contact homology \cite{duality}, and we reinterpret
duality in terms of the cup product in Section~\ref{ssec:duality}.
Though interesting structurally, the cup products that generate the
duality pairing are of no use as invariants.  There exist knots,
however, with nontrivial --- and noncommutative --- cup products that
do not contribute to the duality pairing.  In fact, all of the product
structures produce nontrivial invariants.

\begin{thm} 
  \label{thm:mir-ex}
  There exists an infinite family of knots that are distinguished from
  their Legendrian mirrors by their linearized cohomology rings.  More
  generally, for each $n >2$, there exists an infinite family of knots
  that are distinguished from their Legendrian mirrors by their
  $n^{th}$-order Massey products but not by their $k^{th}$ order
  Massey products for all $k < n$.
\end{thm}

Further, the product structures incorporate all of the information (and
more) from Chekanov's higher-order linearizations.

\begin{thm} 
  \label{thm:a-infty-high-order}
  For all $n > 1$, the $A_n$ structure on $LCH^*_\aug(K)$ is
  strictly stronger than the order $n$ linearized contact
  (co)homology.
\end{thm}

Finally, we can reinterpret a result of the fourth coauthor
\cite{duality} in terms of the cup product operation.

\begin{thm} \label{thm:cup-duality} For every Legendrian knot $K$ in
  the standard tight contact structure on $\rr^3$ and every
  augmentation $\aug$ of its contact homology DGA, there is an element
  $\kappa \in LCH^\aug_1(K)$ and an element $c\in LCH^1_\aug(K)$ such
  that $\langle c, \kappa \rangle = 1$ and the pairing
  \[
  \overline{LCH}_\aug^k\otimes \overline{LCH}_\aug^{-k} \to \zz_2: [a]
  \otimes [b] \mapsto \langle [a]\cup[b], \kappa \rangle
  \]
  is symmetric and non-degenerate, where $\overline{LCH}_\aug^*$ is a
  complement of the span of $c$.
\end{thm}



We leave as open problems whether or not the Massey product structure
on $LCH^*_\aug(K)$ determines the third order linearized contact
(co)homology and whether or not the higher order linearized contact
(co)homologies are stronger invariants than the (first order)
linearized contact (co)homology.

The rest of the paper is organized as follows: after reviewing some
basic definitions in Section~\ref{sec:background}, we define the
$A_\infty$ and product structures in Section~\ref{sec:prod}.  We show
that the product structures are effective invariants in
Section~\ref{sec:invariants} by proving Theorem~\ref{thm:mir-ex}.  We
also establish Theorem~\ref{thm:cup-duality} in this section.
Finally, we prove Theorem~\ref{thm:a-infty-high-order} in
Section~\ref{sec:higher-order}.

{\em Acknowledgements}: The authors thank Jim Stasheff for several
helpful discussions. The first author was supported by an REU  program funded by the
Georgia Institute of Technology. The second author was partially supported by NSF grants
DMS-0804820.  The third and fifth authors were supported as
undergraduate summer research students by the Haverford College
faculty support fund.

\section{Background and Notation}
\label{sec:background}


We refer the reader to the survey article \cite{etnyre:knot-intro} or
the first chapter of the text \cite{lecnotes} for the basic notions of
Legendrian knot theory.  In particular, we assume that the reader is
familiar with the Lagrangian ($xy$) and front ($xz$) projections (and
the resolution procedure that generates a Lagrangian projection from a
front projection by smoothing left cusps, turning right cusps into
loops, and making all crossings be of the form in
Figure~\ref{fig:reeb-and-corner}) of a Legendrian knot in the standard
contact $(\rr^3, \xi_0 = \ker \alpha_0)$.

\subsection{Legendrian Contact Homology}
\label{ssec:lch}

In this section, we sketch the definition of Legendrian contact
homology, which is the homology of the Chekanov-Eliashberg
differential graded algebra (DGA).  See Chekanov's original paper
\cite{chv}, the paper \cite{ens}, or the expository works
\cite{etnyre:knot-intro, lecnotes} for more details.  

To define the Chekanov-Eliashberg DGA $(\alg_K, \df)$ of a Legendrian
knot $K$, we begin with the underlying algebra $\alg_K$.  Number the
crossings of the Lagrangian projection of $K$ from $1$ to $n$, and let
$A$ be the vector space over $\zz_2$ generated by the labels $\{q_1,
\ldots, q_n\}$.  Define the algebra $\alg_K$ to be the unital tensor
algebra over $A$, i.e.,
\begin{equation}
  \alg_K =  \bigoplus_{k=0}^\infty A^{\otimes k}.
\end{equation}
We sometimes denote $\alg_K$ as $\alg(q_1,\ldots, q_n)$ when we want to
emphasize the generating set for the algebra.

The generators $q_i$ are graded by a Conley-Zehnder index that takes
values in $\zz_{2r(K)}$; the grading then extends naturally to all of
$\alg_K$.  Specifically, assume that at all crossings of $\pi_l(K)$,
the strands meet orthogonally. Given a generator $q_i$, choose a path
$\gamma_i$ inside $\pi_l(K)$ that starts on the overcrossing at $i$
and ends at the undercrossing.  Then define the grading $|q_i|$ to be:
\begin{equation}
  |q_i| \equiv 2r(\gamma_i) - \frac{1}{2}  \mod 2r(K).
\end{equation}

Finally, we need to define the differential $\df$ on the generators of
$\alg_K$; it extends to the full algebra via linearity and the Liebniz
rule.  First, decorate the sectors near every crossing of $\pi_l(K)$
with positive and negative signs --- called Reeb signs --- as in
Figure~\ref{fig:reeb-and-corner}(a).  To find $\df q_i$, let
$\Delta(q_i)$ be the set of immersed disks (modulo smooth
reparametrization) whose boundary lies in $\pi_l(K)$.  Further
stipulate that the disks have convex corners (see
Figure~\ref{fig:reeb-and-corner}(b)) such that the corner covers a
positive Reeb sign at the crossing $i$ and negative Reeb signs at all
other corners (it is possible that there are no other corners).
Finally, each disk in $\Delta(q_i)$ contributes a term to $\df q_i$
consisting of the product of the generators associated to its negative
corners, taken in counterclockwise order starting after $i$. Note that
the DGA $(\alg_{\overline{K}},\df)$ for the Legendrian mirror
$\overline{K}$ has the same generators as those for $K$, but the order
of each term in the differential is reversed.

\begin{figure}
  \relabelbox \small {
    \centerline {\epsfbox{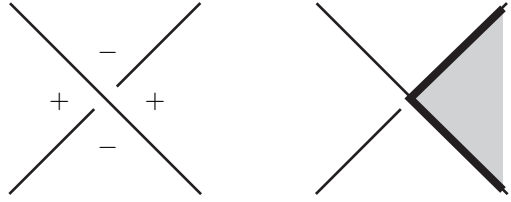}}} 
  \relabel{1}{$-$} 
  \relabel {2}{$+$} 
  \relabel{3}{$+$} 
  \relabel {4}{$-$}
  \endrelabelbox
  \caption{(a) The Reeb signs near a crossing of $\pi_l(K)$. (b) A
    convex corner.}
  \label{fig:reeb-and-corner}
\end{figure}

That this definition produces an invariant of Legendrian knots was
proven by Chekanov.

\begin{thm}[Chekanov \cite{chv}] \label{thm:dga} The differential
  $\df$ has degree $-1$ and satisfies $\df^2=0$.  The Legendrian
  contact homology $H_*(\alg_K,\df)$ is invariant under Legendrian
  isotopy.
\end{thm}

In fact, Chekanov proved something more subtle: the ``stable tame
isomorphism'' class of $(\alg_K, \df)$ is an invariant.  
We recall the definition of stable tame isomorphism as it will be used below.

  A graded chain isomorphism
  $$\phi: \alg(q_1, ..., q_n) \longrightarrow \alg(q'_1,\ldots, q'_n)$$
  is {elementary} if there is some $j \in \{1, \ldots, n\}$
  such that
  \begin{equation} 
    \phi(q_i) = 
    \begin{cases}
      q'_i, & i \neq j \\
      q'_j + u, & i=j \text{ where }u \in
      \alg(q'_1,\ldots,q'_{j-1},q'_{j+1},\ldots,q'_n).
    \end{cases}
  \end{equation}
A composition of elementary isomorphisms is called {tame}.
The degree $i$-{stabilization} $S_i(\alg(q_1, \ldots, q_n))$
  of $\alg(q_1, \ldots, q_n)$ is defined to be $\alg(q_1, \ldots, q_n,
  e^i_1, e^i_2)$.  The grading and the differential are inherited from
the original algebra and by setting  $|e_1| = i$, $|e_2| = i-1$, $\df e_1 = e_2$,
and $\df e_2 = 0$.

  Two differential algebras $(\alg, \df)$ and $(\alg', \df')$ are 
  {stably tame isomorphic} if after each algebra has been stabilized some number 
  of times they become  tame isomorphic by a map that is also a chain map.


\subsection{Linearized Contact Homology and Cohomology}
\label{ssec:lin}

As it stands, it is difficult to use Legendrian contact homology in
practical applications, as it is the homology of an infinite
dimensional noncommutative algebra with a nonlinear differential.  To
find a more amenable invariant, we use Chekanov's linearization
technique.  To do this, we break up the differential on $A$ into its
components:
\begin{equation}
  \df_k : A \to A^{\otimes k}.
\end{equation}
If it were true that the constant term of the differential vanished,
i.e. if $\df_0 = 0$, then the fact that $\df^2 = 0$ would imply that
$\df_1^2 = 0$.  In particular, if $\df_0 = 0$, then $(A, \df_1)$ is a
finite-dimensional chain complex with easily computable homology.

It is rarely true, however, that $\df_0 = 0$.  To remedy this
situation, consider graded algebra maps $\aug: \alg_K \to
\zz_2$ that satisfy:
\begin{enumerate}
\item $\aug(1) = 1$, and
\item $\aug \circ \df = 0$.
\end{enumerate}
These maps are called \dfn{augmentations}.  They do not always
exist --- see, for example, \cite{fuchs-ishk, rutherford:kauffman,
  rulings} --- but when they do, they allow us to linearize the
Chekanov-Eliashberg DGA.  To see how, consider the graded isomorphism
$\phi^\aug: \alg_K \to \alg_K$ defined by $\phi^\aug (q_i) = q_i +
\aug(q_i)$.  This map defines a new differential $\dfe = \phi^\aug \df
(\phi^\aug)^{-1}$; it is easy to check that $\dfe_0 = 0$.  Thus, for
each augmentation $\aug$ of $(\alg_K, \df)$, there is a chain complex
$(A, \dfe_1)$.  This is called the \dfn{linearized chain complex
  with respect to \aug}.  There is also a cochain complex $(A^*,
\dde)$, where $A^*$ has a basis $\{p_1, \ldots, p_n\}$ that is dual
to $\{q_1, \ldots, q_n\}$ and $\dde$ is the adjoint of $\dfe_1$.
The homologies of these complexes are called the \dfn{linearized
  contact (co)homologies} and are denoted by $LCH^\aug_*(K)$ and
$LCH_\aug^*(K)$.

Chekanov extended the definition of the linearized (co)chain complex
to include higher-order pieces of the differential.  The
$n^{\text{th}}$\dfn{-order linearized chain complex with respect to
  \aug} is given by the graded vector space of chains
\[A_{(n)} = \bigoplus_{i=1}^\infty A^{\otimes i} /
\bigoplus_{i=n+1}^\infty A^{\otimes i}\] together with the
differential $\dfe_{(n)}$ induced from the quotient. Notice that
$\dfe$ does indeed descend to the quotient since $\dfe_0=0$, so $\dfe$
cannot decrease the length of a tensor. The $n^{\text{th}}$-order
cochain complex is defined by taking duals and adjoints, as usual. The
homologies of these complexes are called the
$n^{\text{th}}$\dfn{-order linearized contact (co)homologies} and are
denoted by $LCH^\aug_*(K,n)$ and $LCH_\aug^*(K,n)$.

Chekanov proved that the set of all linearized (co)homologies taken
over all possible augmentations is a Legendrian knot invariant; this
set is called the \dfn{linearized (co)homology invariant} of $K$.
Invariance also holds for the higher order linearized homologies.

\begin{rem} \label{rem:lin-pf}
  The proof relies on two facts that were proved in \cite{chv}: first,
  the linearized invariant does not change under stabilizations of the
  Chekanov-Eliashberg DGA.  Second, given a tame isomorphism $\psi:
  (\alg, \df) \to (\alg', \df')$ and an augmentation $\aug'$ of
  $(\alg', \df')$, the composite map $\phi^{\aug'} \psi$ factors as
  $\overline{\psi} \phi^\aug$, where $\aug$ is an augmentation of
  $\alg$ and $\overline{\psi}$ does not reduce the lengths of words in
  $\alg$.  The map $\overline{\psi}$ is a DGA isomorphism between
  $(\alg, \df^\aug)$ and $(\alg', (\df')^{\aug'})$, and hence
  restricts to an isomorphism between the linearized complexes $(A,
  \df^\aug_1)$ and $(A', (\df')^{\aug'}_1)$.
\end{rem}

\section{$A_\infty$-algebras and Product Structures}
\label{sec:prod}

As mentioned in the introduction, invariant product structures can be
defined on the linearized cohomology invariant by using higher-order
terms in the differential $\df$.  In fact, we shall see that the
linearized cochain complex carries the structure of an $A_\infty$
algebra, and that the $A_\infty$ structure on the cochain complex
induces an invariant $A_\infty$ structure on the linearized contact
cohomology.

\subsection{$A_\infty$ Algebras and Massey Products}
\label{ssec:ainfty}

An \dfn{$A_n$-algebra} over $\zz_2$ is a graded vector space $V$ over
$\zz_2$ together with a sequence of graded maps $\m = \{m_k:V^{\otimes
  k}\to V\}_{1 \leq k \leq n}$ of degree $1$ satisfying:
\begin{equation}\label{main-a-infty}
  \sum_{i+j+k=l} m_{i+1+k}\circ (1^{\otimes i}\otimes 
  m_j\otimes 1^{\otimes k})=0
\end{equation}
for all $1 \leq l \leq n$. An \dfn{$A_\infty$-algebra} is the obvious
generalization to an infinite sequence of maps.  An $A_\infty$-algebra
structure induces $A_n$ structures for all $n \geq 1$. Notice that
Equation~\eqref{main-a-infty} for $l=1$ is $m_1\circ m_1=0,$ which
implies that $m_1$ is a co-differential on $V$.  From now on, we
denote it by $\delta$. The cohomology of $(V,\delta)$ is denoted
$H^*(V).$ When we take $l=2$ in Equation~\eqref{main-a-infty}, we get:
\[
\delta m_2(a,b) = m_2(\delta a, b)+m_2(a,\delta b)
\]
for all $a,b\in V.$ Thus, $m_2$ descends to a well defined product
$\mu_2$ on $H^*(V)$. We see this product is associative using
Equation~\eqref{main-a-infty} when $l=3$:
\begin{equation} \label{eqn:m3}
  \begin{aligned}
    m_2(a,m_2(b,c))&+m_2(m_2(a,b),c)=\delta m_3(a,b,c)\\ & +m_3(\delta
    a, b,c)+m_3(a,\delta b, c)+ m_3(a,b,\delta c).
  \end{aligned}
\end{equation}
Thus, given an $A_\infty$ algebra $(V, \m)$, we obtain an ordinary
associative algebra $(H^*(V), \mu_2)$.

\begin{rem}
  Usually, the $A_\infty$ algebra map $m_k$ is taken to have degree
  $2-k$ instead of degree $1$. Our maps $m_k$ should be thought of as
  degree $1$ maps on the suspension $A^* = SV$ induced by degree $2-k$
  maps $\widetilde{m}_k$ defining a conventionally-graded $A_\infty$
  algebra $(V, \widetilde{\bf{m}})$.  Similar comments apply to the
  definition of $A_\infty$ morphisms, which are usually taken to have
  degree $n-1$ instead of degree $0$.
\end{rem}

If we try to define a full $A_\infty$ structure on $H^*(V)$ by simply
letting the maps $m_k$ descend to cohomology, we run into trouble
already at $k=3$, as Equation~\eqref{eqn:m3} shows that $m_3(a,b,c)$
is not necessarily a cycle even if $a$, $b$, and $c$ are.  We can
proceed in one of two ways: first, following Stasheff
\cite{stasheff:h-space-2}, we can (partially) define a triple product
on $H^*(V)$ as follows: given $[a],[b],[c]\in H^*(V)$, suppose that
$\mu_2([a],[b])=[0]=\mu_2([b],[c])$. Let $\delta x = m_2(a,b)$ and let
$\delta y = m_2 (b,c)$. Then we see that \[m_3(a,b,c) + m_2(a,y) +
m_2(x,c)\] is a cocycle.  Since $x$ and $y$ are only defined up to the
addition of cocycles, we get a well-defined element
\[ \{[a],[b],[c]\} \in \tilde{H}^*(V) = \frac{H^*(V)}{\img
  (\mu_2([a],\cdot))+\img(\mu_2(\cdot,[c]))}.
\]
This triple product is called a \dfn{Massey product}. 

It is possible to inductively define higher-order Massey products on
$H^*(V)$ using the $A_\infty$ structure. Given $[a_1], \ldots, [a_n]
\in H^*(V)$, suppose that the product $\{[a_i], \ldots, [a_j]\}$ is
defined and equal to zero modulo the successive images of all
lower-order Massey products for all $1 \leq i < j \leq n$.  Following
the order $3$ case in \cite{stasheff:h-space-2}, we define:
\[
\{[a_1], \ldots, [a_n]\}= \left[\sum_{k=2}^n \sum_{0 \leq i_1 < \cdots
    <
    i_{k-1} <
  n} m_k(b_{1,i_1}, b_{i_1+1,i_2}, \ldots, b_{i_{k-1}+1,n}) \right],
\]
where $b_{lm} \in V$ has been inductively defined by:
\begin{align*}
  [b_{mm}] &= [a_m], \\
  \delta b_{lm} &= \sum_{k=2}^{m-l+1} \sum_{l \leq i_1 < \cdots < i_{k-1} <
    m} m_k(b_{l,i_1}, b_{i_1+1,i_2}, \ldots, b_{i_{k-1}+1,m}).
\end{align*}
It is straightforward but tedious to check using the definition of the
$b_{km}$ and the defining $A_\infty$ equation \eqref{main-a-infty}
that the higher-order Massey product is indeed a cocycle and is
well-defined modulo the successive images of the lower-order Massey
products in $H^*(V)$.

The Massey products have the practical advantage of computability, as
we shall see, but the theoretical disadvantage of being only partially
defined.  The second way forward is to try to define a full $A_\infty$
structure on $H^*(V)$.  To do this, we need the notion of a
\dfn{morphism} of $A_\infty$-algebras; there are obvious analogs for
$A_n$ algebras, and any $A_\infty$ morphism induces $A_n$ morphisms
for all $n \geq 1$. An $A_\infty$ morphism $\boldsymbol{\phi}:(V, \m)
\to (W, \n)$ is a collection of degree $0$ linear maps
$\phi_n:V^{\otimes n}\to W$ that satisfy
\begin{equation}\label{a-infty-morphism}
  \sum_{i+j+k=n} \phi_{i+1+k}\circ(1^{\otimes i}\otimes m_j\otimes
  1^{\otimes k})
  =\sum_{\substack{1\leq r\leq n \\ i_1+\ldots+i_r=n}} n_r\circ 
  (\phi_{i_1}\otimes\cdots\otimes\phi_{i_r}).
\end{equation}
Notice that this equation implies that $\phi_1:V\to W$ commutes with
the codifferentials on $V$ and $W$, and hence induces a map on
cohomology. The morphism $\boldsymbol{\phi}$ is called an
\dfn{$A_\infty$ quasi-isomoprhism} if $\phi_1$ induces an isomorphism
on the cohomology.

Equation~\eqref{a-infty-morphism} for $n=2$ says that 
\[
\phi_1\circ m_2 + \phi_2(\dd\otimes 1 + 1\otimes \dd)=
n_2\circ(\phi_1\times \phi_1) + \dd\circ \phi_2.
\]
Thus, on the level of cohomology,
\[
\phi_1 m_2([a],[b])=m_2(\phi_1[a], \phi_1[b]).
\]
In other words, $\phi_1$ preserves the product structure on
cohomology.  One may easily check that
Equation~\eqref{a-infty-morphism} for $n>2$ implies that $\phi_1$ will
preserve the Massey product and higher order product structures on the
cohomology as well.

We now return to the discussion of defining an $A_\infty$ structure on
$H^*(V)$.  The relevant result is the Minimal Model Theorem, which we
shall discuss in more detail in Section~\ref{sec:higher-order}.

\begin{thm}[Kadeishvili \cite{kadeishvili}] \label{thm:kad} If $(V,
  \m)$ is an $A_\infty$ algebra over a field, then its homology
  $H^*(V)$ also possesses an $A_\infty$ structure $\bmu$ such that
  $\mu_1 = 0$, $\mu_2$ is induced from $m_2$ as described above, and
  there exists an $A_\infty$ quasi-isomorphism $\boldsymbol{\phi}:
  (H^*(V), \bmu) \to (V, \m)$.  The $A_\infty$ structure on $H^*(V)$
  is unique up to $A_\infty$ quasi-isomorphism.
\end{thm}

\subsection{The $A_\infty$ Structure on the Linearized Cochain Complex}
\label{ssec:ainfty-lch}

The reason for discussing $A_\infty$-algebras is the following
proposition.

\begin{prop} \label{prop:lch-is-ainfty} For each augmentation $\aug$,
  the Legendrian contact homology DGA $(\alg, \df)$ induces an
  $A_\infty$ structure on the linearized cochain complex $(A^*,
  \dde)$.
\end{prop}

\begin{proof}
  Denote by $m^\aug_k$ the adjoint of $\dfe_k:A\to A^{\otimes k}$ for
  $k \geq 1$. Expanding the equation $(\dfe)^2=0$ using $\dfe=\sum
  \dfe_i$ and looking at the term with image in $A^{\otimes n}$ gives:
  \[
  \sum_{i+j+k=n} (1^{\otimes i}\otimes \dfe_j\otimes 1^{\otimes
    k})\circ \dfe_{i+1+k}=0.
  \]
  Dualizing yields Equation~\eqref{main-a-infty}.  That $m^\aug_k$ has
  degree $1$ follows from the fact that $\dfe_k$ has degree $-1$.
\end{proof}

\begin{ex}\label{ex:a-infty}
  Let $K$ be the Legendrian trefoil shown in Figure~\ref{fig:trefoil}.
  \begin{figure}[ht]
    \relabelbox \small {
      \centerline {\epsfbox{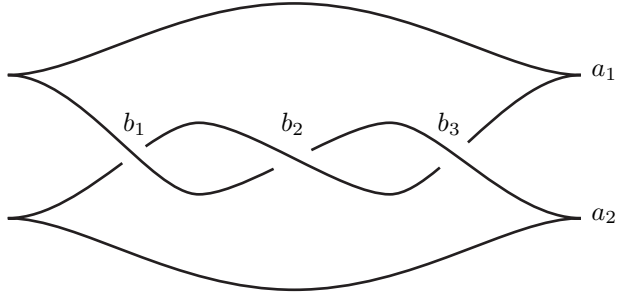}}} 
    \relabel{1}{$b_1$} 
    \relabel {2}{$b_2$} 
    \relabel{3}{$b_3$} 
    \relabel {5}{$a_1$}
    \relabel{6}{$a_2$} 
    \endrelabelbox
    \caption{Legendrian right handed trefoil knot.}
    \label{fig:trefoil}
  \end{figure}
  We label the Reeb chords $a_1,a_2, b_1,b_2$ and $b_3$ as shown in
  the figure. One may easily compute that $|a_i|=1$ and $|b_i|=0.$ In
  addition, we have:
  \begin{align*}
    \partial a_1&=1+b_1+b_3+ b_1b_2b_3\\
    \partial a_2&=1+b_1+b_3+ b_3b_2b_1\\
    \partial b_i&= 0.
  \end{align*}
  There are five different augmentations of this differential graded
  algebra; let us consider the augmentation $\aug$ that sends $b_3$ to
  $1$ and all other generators to $0$. The augmented differential is:
  \begin{align*}
    \dfe a_1&=b_1+b_3+ b_1b_2+ b_1b_2b_3\\
    \dfe a_2&=b_1+b_3+ b_2b_1+ b_3b_2b_1\\
    \dfe b_i&= 0.
  \end{align*}
  Thus, if we denote the dual of $a_i$ again by $a_i$, and similarly
  for $b_i$, the associated $A_\infty$-structure is:
  \begin{align*}
    m_1(a_1)&= 0 &  m_2(b_1,b_2)&=a_1\\
    m_1(a_2) &= 0 & m_2(b_2,b_1)&=a_2 \\
    m_1(b_1)&=a_1+a_2 & & \\
    m_1(b_2)&=0 & m_3(b_1,b_2,b_3)&=a_1 \\
    m_1(b_3)&=a_1+a_2 & m_3(b_3,b_2,b_1)&=a_2
  \end{align*} 
  All other possible $m_2$ and $m_3$ products are $0$, as are the
  $m_i$ for $i \geq 4$.  The $A_\infty$-algebras associated to the
  other four augmentations can be similarly computed.
\end{ex}

Like the set of linearized cohomologies, the set of $A_\infty$
structures on the linearized cochain complexes is an invariant.

\begin{thm}\label{thm:a-infty}
  If the DGA $(\alg, \df)$ of a Legendrian knot has a set of
  augmentations $\mathcal{E}$, then the set of all quasi-isomorphism
  types of the $A_\infty$-algebras
  \[\bigl\{(A^*, \m^\aug)\bigr\}_{\aug \in \mathcal{E}}
  \]
  is invariant under Legendrian isotopy of the knot.
\end{thm}

Theorem~\ref{thm:kad} shows that there is an induced $A_\infty$
structure on the linearized cohomology, and that it is also an
invariant.

\begin{cor}
  The following structures are invariant of a Legendrian knot up to
  Legendrian isotopy:
  \begin{enumerate}
  \item The set of linearized cohomology rings together with their
    higher order product structures.
  \item The set of $A_\infty$ algebras
    $\bigl\{(LCH^*_\aug(K), \bmu^\aug)\bigr\}_{\aug \in
      \mathcal{E}}.$
  \end{enumerate}
\end{cor}

\begin{proof}[Proof of Theorem~\ref{thm:a-infty}]
  As in Remark~\ref{rem:lin-pf}, it suffices to show that if $(\alg,\df)$
  and $(\alg', \df')$ are stable tame isomorphic DGAs such that
  $\df_0=0=\df'_0$ and the tame isomorphism $\psi$ between the
  stabilizations satisfies $\psi_0=0$, then their associated
  $A_\infty$-algebras are $A_\infty$-quasi-isomorphic.  We shall check
  that the statement is true for tame isomorphisms and stablizations.
  
  First, let $\psi:\mathcal{A}\to \mathcal{A}'$ be a tame isomorphism
  satisfying the conditions above. The component of $\psi\circ
  \df=\df'\circ \psi$ applied to $a\in A$ 
  written in terms of the components $\df_i$, $\df'_i$ 
  and $\psi_i$ is 
  \[
  \sum_{i+j+k=n} (1^{\otimes i}\otimes \df_j\otimes 1^{\otimes k})
  \circ \psi_{i+1+k}= \sum_{1\leq r\leq n, i_1+\ldots+i_r=n}
  (\psi_{i_1}\otimes\cdots\otimes \psi_{i_r})\circ \df_r.
  \]
  Setting $\phi_n$ equal to the dual of $\psi_n$, we clearly see that
  Equation~\eqref{a-infty-morphism} is dual to this
  equation. Moreover, as we know a tame isomorphism of differential
  graded algebras induces an isomorphism on linearized cohomology, we
  see that the collection of maps $\boldsymbol{\phi} = \{\phi_n\}$ is
  an $A_\infty$-quasi-isomoprhism.

  Now consider $\psi: \alg \to \alg' = S(\alg)$ to be the inclusion of
  $\alg$ into a stabilization.  Specifically, let
  $A'=A\oplus\zz_2\langle a,b\rangle$ where $\df'a=b$ and $\df c=\df'
  c$ for $c\in A$. Note that $\psi_1$ is the inclusion map and
  $\psi_n=0$ for $n>1.$ The result clearly follows.
\end{proof}

\section{Product Structures as Invariants}
\label{sec:invariants}

In this section, we consider the products induced by $A_\infty$
structure on the linearized cochain complex.  That is, we study the
cup and Massey products on the linearized contact cohomology of a
Legendrian knot in more detail, prove that they are nontrivial
invariants, and relate the cup product to the duality of
\cite{duality}.

Throughout this section, we let $(\mathcal{A}_K,\df)$ be a
differential graded algebra associated to a Legendrian knot $K$ in
$\rr^3$ with its standard contact structure. Let
$\aug:\mathcal{A}_K\to\zz_2$ be an augmentation and let
$\dfe:\mathcal{A}_K\to \mathcal{A}_K$ be the associated differential
with $\dfe_0=0.$

\subsection{The Cup Product}
\label{ssec:cup}

We summarize the discussion of the $\mu_2$ product from the previous
section as follows:
\begin{cor} \label{cor:cup-exist}
  There is an associative product on the linearized contact cohomology
  of $K$ given by the $\mu_2$ product:
  \[
  LCH^{k}_\aug(K)\otimes LCH^{l}_\aug(K)\to LCH_\aug^{k+l+1}(L):
  [a] \otimes [b] \mapsto [a]\cup[b].
  \]
  Moreover, the set of all linearized contact cohomology rings is an
  invariant Legendrian isotopy.
\end{cor}

\begin{ex}\label{ex:trfoilprod}
  Consider the Legendrian trefoil $K$ from Figure~\ref{fig:trefoil}
  again. We computed the $A_\infty$-algebra structure in
  Example~\ref{ex:a-infty} above. From there, we easily see that
  $LCH^1_{\aug}(K)=\zz_2$ generated by $a=[a_1]=[a_2]$ and
  $LCH^0_{\aug}=\zz_2\oplus\zz_2$ generated by $b=[b_2]$ and
  $c=[b_1+b_3].$ Moreover we easily see that:
  \[
  b\cup c=a\quad c\cup b=a
  \]
  and all other products are zero. Note that the product structure
  here is commutative. This is not the case in general.
\end{ex}

We are now ready to prove the first part of Theorem~\ref{thm:mir-ex},
i.e.\ that the set of linearized contact cohomology rings is a
nontrivial invariant and stronger than the linearized contact
cohomology groups.

\begin{figure}[ht]
  \relabelbox \small {
    \centerline {\epsfbox{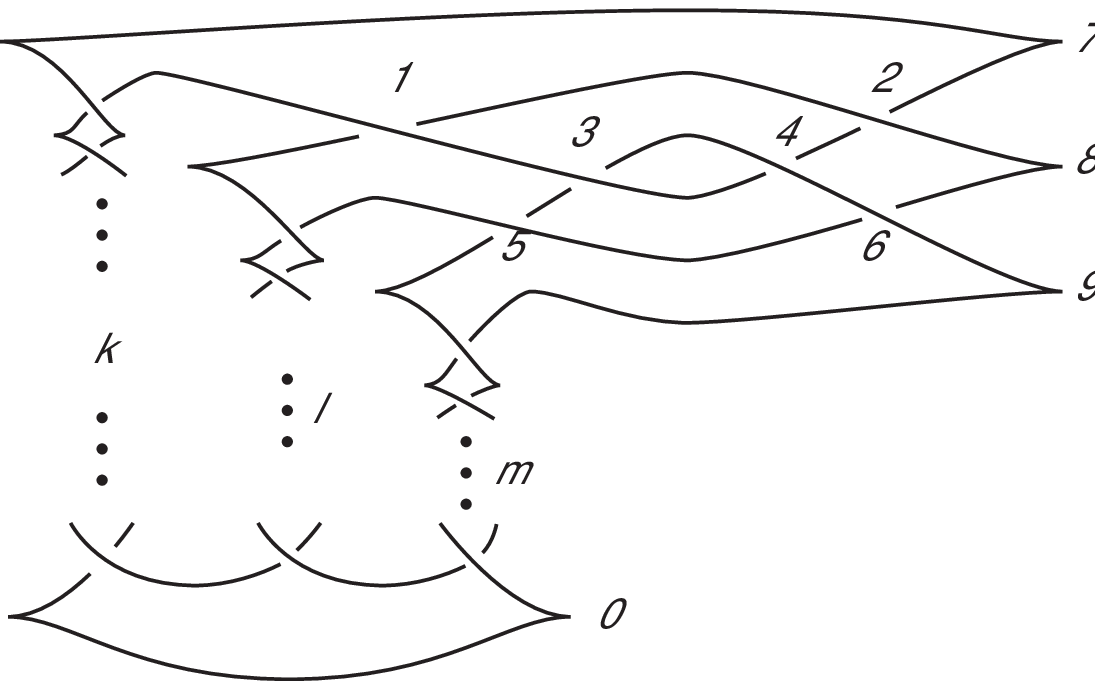}}} 
  \relabel{1}{$a_1$} 
  \relabel {2}{$a_2$} 
  \relabel{3}{$b_1$} 
  \relabel {4}{$b_2$}
  \relabel{5}{$c_1$} 
  \relabel{6}{$c_2$} 
  \relabel {7}{$t_1$} 
  \relabel{8}{$t_2$} 
  \relabel {9}{$t_3$}
  \relabel{0}{$t_0$}
  \relabel{k}{$k$} 
  \relabel{l}{$l$} 
  \relabel{m}{$m$} 
  \endrelabelbox
  \caption{This knot is distinguished from its Legendrian mirror by
    its cohomology ring. The crossings along the left most, center and
    right most legs are denoted, respectively, by $x_i,$ $y_i,$ and
    $z_i.$ Similarly the crossings coming from resolving the right
    cusps along these legs are denoted by $t^x_i, t^y_i$ and $t^z_i$
    respectively.}
  \label{fig:cup-ex}
\end{figure}

\begin{proof}[Proof of the first part of Theorem~\ref{thm:mir-ex}] For infinitely many
  choices of $k,l,m$, the Legendrian knot in Figure~\ref{fig:cup-ex}
  is not Legendrian isotopic to its Legendrian mirror.  The knot and
  its mirror have the same classical invariants and the same
  linearized cohomology, but different linearized cohomology rings.

  To see this, we first compute the gradings of the generators:
  \begin{align*}
    |a_1| = -|a_2| &= k-l-1 \\
    |b_1| = -|b_2| &= k-m-1 \\
    |c_1| = -|c_2| &= l-m-1 \\
    |t_i| = |t^x_i|=|t^y_i|=|t^z_i| &= 1 \\
    |x_i| = |y_i| = |z_i| &= 0
  \end{align*}
  For infinitely many choices of $k,l,m$, the gradings in each row will be
  distinct.

  The differential has the following form:
  \begin{align*}
    \df a_1 &= 0 &   \df t_1 &= 1+ x_1 (1+a_1a_2+b_1b_2) \\
    \df a_2 &= y_1c_1b_2 & \df t_2 &= 1+ (1+a_2a_1)y_1(1+c_1c_2) \\
    && \df t_3 &= 1+ (1+b_2b_1 + c_2c_1)z_1 \\
    \df b_1 &= a_1y_1c_1 & \df t_0 &= 1 + x_{k+1} y_{l+1} z_{m+1} \\
    \df b_2 &= 0 \\
    && \df t^x_i &= 1+ x_ix_{i+1} \\
    \df c_1 &= 0 & \df t^y_i &= 1+ y_iy_{i+1} \\ \df c_2 &= b_2a_1y_1 & \df
    t^z_i &= 1+ z_iz_{i+1}
  \end{align*}

  Recall that for the Legendrian mirror of $K$, the ordering of the
  generators in the differential are all reversed.  In either case,
  since all but at most one of the $a_i$, $b_i$, or $c_i$ have nonzero
  grading, then there is a unique augmentation $\aug$ that sends the
  $x_i$, $y_i$, and $z_i$ to $1$ and all other generators to $0$.

  The linearized codifferential $\dde$ of all generators $a_i$, $b_i$,
  and $c_i$ vanishes (where we again abuse notation and identify a
  generator with its dual), as does the linearized codifferential of
  the generators coming from the right cusps.  The linearized
  codifferentials of the $x_i$, $y_i$, and $z_i$ generators are sums
  of ``adjacent'' right cusp generators, so they show that the
  generators coming from the right cusps are all equal in cohomology.
  Thus, we have the following compuation:
  \begin{equation}
    LCH^k_\aug(K) = \langle [a_i], [b_i], [c_i], [t] \rangle.
  \end{equation}
  The nontrivial cup products are:
  \begin{align*}
    [a_1]\cup[a_2] = [a_2] \cup [a_1] &= [t] & [c_1]\cup[b_2] &=
    [a_2] \\
    [b_1] \cup [b_2] = [b_2] \cup [b_1] &= [t] & [a_1] \cup [c_1] &=
    [b_1] \\
    [c_1] \cup [c_2] = [c_2] \cup [c_1] &= [t] & [b_2] \cup [a_1] &=
    [c_2]
  \end{align*}
  The cup products in the left column will be interpreted as part of a
  Poincar\'e duality pairing in the next section.  The cup products in
  the right column are not symmetric; the first, for example, is a
  nontrivial map from $LCH^{l-m-1}_\aug \otimes LCH^{m-k+1}_\aug \to
  LCH^{l-k+1}_\aug$.  Under the assumption that the generators $a_i$,
  $b_i$, and $c_i$ have distinct gradings, we can then easily see that
  no such nontrivial cup product exists in the cohomology ring of the
  Legendrian mirror.  Hence, the knot $K$ and its Legendrian mirror
  are not Legendrian isotopic.
\end{proof}

\begin{rem}
  There are examples of Legendrian knots with small crossing number
  that have augmentations with noncommutative linearized cohomology
  rings: consider, for example, the mirrors of the knots $8_{21}$,
  $9_{45}$, or $9_{47}$ in Melvin and Shrestha's table
  \cite{melvin-shrestha}.
\end{rem}

\subsection{Duality}
\label{ssec:duality}



We are now ready to prove Theorem~\ref{thm:cup-duality} concerning the
duality in \cite{duality}. We note that Theorem~\ref{thm:cup-duality}
implies the product operation in the ring structure of linearized
contact cohomology is non-trivial, while the first part of
Theorem~\ref{thm:mir-ex} shows that there are non-trivial products
that are not forced by the duality theorem.

\begin{proof}[Proof of Theorem~\ref{thm:cup-duality}]
  As described in \cite{duality}, there is chain complex $(Q_*,
  \df_Q)$ that can be thought of in two ways: first, it is the
  mapping cone for a map $\rho: (A_*, \dfe_1) \to CM_*(S^1; f)$,
  where $CM_*$ is the Morse complex for a Morse function $f$ on $S^1$.
  As noted in \cite{high-d-duality}, the long exact sequence of the
  mapping cone is:
  \begin{equation} \label{eqn:les}
    \cdots \to H_{k+1}(S^1) \to H_k(Q) \to LCH^\aug_k(K)
    \stackrel{\rho_*}{\to} H_{k}(S^1) \to \cdots.
  \end{equation}
  Further, $\rho_*$ is trivial in dimension $0$ and onto in dimension
  $1$; see the discussion after Lemma 4.9 of \cite{duality} or Theorem
  5.5 of \cite{high-d-duality}.
  
  The second perspective on $H_*(Q)$ is that there exists an
  isomorphism $\eta_*: H_k(Q) \to LCH^{-k}_\aug(K)$.  Putting these
  viewpoints together yields the isomorphisms:
  \begin{align*}
      LCH^1_\aug(K) & \simeq LCH_{-1}^\aug(K) \oplus H_0(S^1) &\\
      LCH_1^\aug(K) & \simeq LCH^{-1}_\aug(K) \oplus H_1(S^1) &\\
      LCH^k_\aug(K) & \simeq LCH_{-k}^\aug(K) & k \neq \pm 1.
  \end{align*}
  Let $c$ be the image under $\eta_*$ of a generator of $H_0(S^1)$ and
  define:
  \begin{equation}
    \overline{LCH}_\aug^{1} = \eta_* LCH_{-1}^\aug(K).
  \end{equation}
  Finally, we define $\kappa \in LCH_1^\aug(K)$ to come from
  $H_1(S^1)$.  The main theorem of \cite{high-d-duality} shows that
  $\kappa$ is the unique class that pairs to $1$ with $c$ and pairs to
  $0$ on $\overline{LCH}_\aug^{1}$, and hence agrees with the $\kappa$
  defined in Theorem 5.1 of \cite{duality}.

  The map $\eta_*$ has an inverse $\phi_*$ which comes from a ``cap
  product''.  More specifically, the chain map $\phi(p)$ is
  constructed in \cite{duality} by counting immersed disks with one
  negative corner at $p$, one negative corner at the output $q'$, one
  positive corner at $r$ with $\langle r, \kappa \rangle = 1$ (in that
  counterclockwise order), and possibly other negative augmented
  corners.  Such disks, however, also contribute to the evaluation of
  the product $m_2(p,p')$ on $\kappa$.  Passing to homology, we obtain:
  \begin{equation}
    \langle [p'], \phi_*[p] \rangle = \langle[p] \cup [p'], \kappa \rangle.
  \end{equation}
  Since $\phi_*$ is invertible on $\overline{LCH}^*_\aug$, the pairing
  on the right must be nondegenerate, as desired.

  To see that the pairing is symmetric, notice that we could have
  defined a map $\hat{\phi}$ using disks with one negative corner at
  $p$, one positive corner at $r$ with $\langle r, \kappa \rangle =
  1$, one negative corner at the output $q'$ (in that counterclockwise
  order), and possibly other negative augmented corners.  The induced
  map $\hat{\phi}_*$ also serves as an inverse for $\eta_*$, and hence
  must be the same map as $\phi_*$.  Thus:
  \begin{align*}
    \langle [p] \cup [p'], \kappa \rangle &= \langle [p'], \phi_*[p] 
    \rangle \\
    &= \langle [p'], \hat{\phi}_*[p] \rangle \\
    &= \langle [p'] \cup [p], \kappa \rangle.
  \end{align*}
\end{proof}

\subsection{The Massey Product}
\label{ssec:massey}

In this subsection, we study the Massey product on the linearized
contact cohomology of a Legendrian knot in more detail. Using the same
notation as in Section~\ref{ssec:cup}, we summarize the discussion of
the product from the Subsection~\ref{ssec:ainfty} in the following
corollary:

\begin{cor}
  If $[a],[b]$ and $[c]$ are elements in $LCH^*_\aug(K)$ of degrees
  $r,s$ and $t,$ respectively such that
  \[
  [a]\cup [b]=0=[b]\cup [c]
  \]
  then there is a well defined element 
  \[
  \{[a],[b],[c]\}\in \frac{LCH^{r+s+t+1}_\aug(K)}{\bigl(\img
      (\mu^\aug_2([a],\cdot))+\img(\mu^\aug_2(\cdot,[c]))\bigr)\cap
    LCH^{r+s+t+1}_\aug(K)}
  \]
  given by \[[m^\aug_3(a,b,c) + m^\aug_2(a,y) + m^\aug_2(x,c)],\]
  where $\delta^\aug x = m^\aug_2(a,b)$ and $\delta^\aug y = m^\aug_2
  (b,c)$.
\end{cor}

\begin{ex}
  Consider the Legendrian trefoil $K$ from Figure~\ref{fig:trefoil}
  again. We computed the $A_\infty$-algebra structure in
  Example~\ref{ex:a-infty} and the product structure in
  Example~\ref{ex:trfoilprod}. Even though $m_3^{\aug}\not=0$, one
  may easily check that all Massey products are trivial in this
  example.
\end{ex}

Notice that if one wants to compare the Massey product structures on
the linearized contact cohomologies of two Legendrian knots one must
first have an isomorphism of their cohomology rings (that is, an
isomorphism that preserves the product structure). The Massey
product can be non-trivial and distinguish Legendrian knots that are
not distinguished by their linearized contact cohomology ring
structures.

\begin{proof}[Proof of second part of Theorem~\ref{thm:mir-ex}]
  The Legendrian knot $K$ in Figure~\ref{fig:massey-ex} is not
  isotopic to its Legendrian mirror. The two knots can be
  distinguished using the Massey products on the linearized contact
  cohomology but not by their linearized contact cohomology rings.
  \begin{figure}[ht]
  \relabelbox \small {
  \centerline {\epsfbox{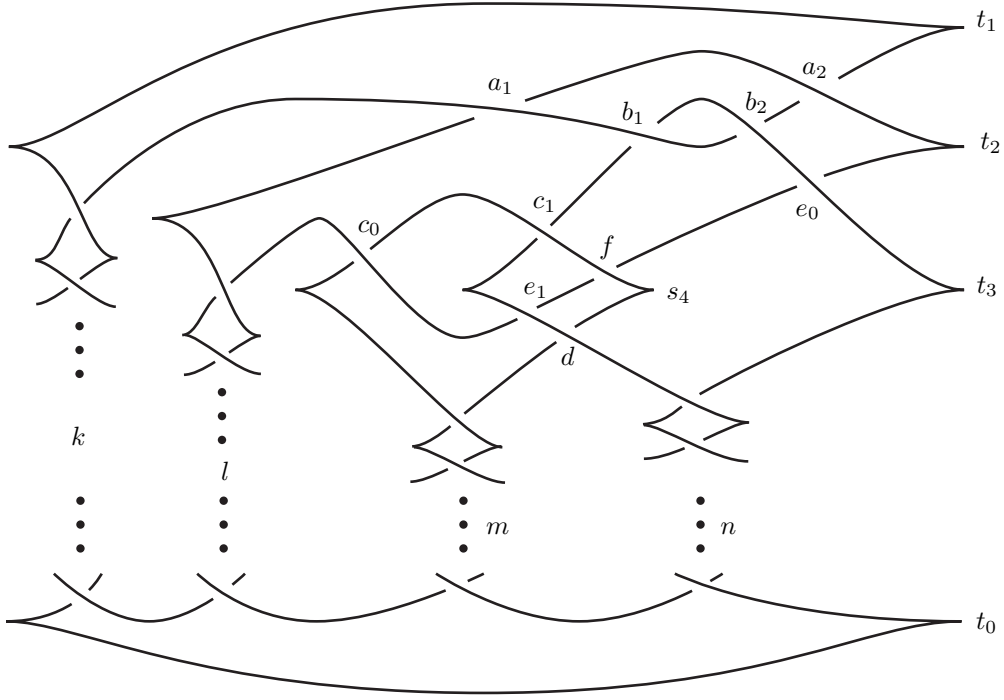}}} 
  \relabel{1}{$a_1$} 
  \relabel {2}{$a_2$} 
    \relabel{3}{$b_1$} 
  \relabel {4}{$b_2$}
    \relabel{5}{$e_1$} 
      \relabel{6}{$e_0$} 
  \relabel {7}{$c_0$} 
    \relabel{8}{$f$} 
  \relabel {9}{$c_1$}
  \relabel{0}{$s_4$}
   \relabel{a}{$d$} 
  \relabel {b}{$t_1$} 
    \relabel{c}{$t_2$} 
  \relabel{d}{$t_3$}
  \relabel{e}{$t_0$}
    \relabel{k}{$k$} 
     \relabel{l}{$l$} 
      \relabel{m}{$m$}
      \relabel{n}{$n$} 
  \endrelabelbox
  \caption{This knot is distinguished from its Legendrian mirror by
    its Massey products. The crossings along the left most, center left, center right and right most legs are denoted, respectively,  by $x_i,$ 
    $y_i, z_i$ and $w_i.$ Similarly the right cusps along these legs are denoted by $t^x_i, t^y_i, t^z_i$ and $t^w_i$ respectively.}
        \label{fig:massey-ex}
\end{figure}

  To see this, we first compute the gradings of the generators:
  \begin{align*}
    &|a_1| = -|a_2| = k-l-1 & &&
    &|b_1| = -|b_2| = k-n-1 \\
    &|c_1| = -|d| = m-n & &&
    &|c_0| = -|f| = l-m-1 & \\
    &|e_0| = |e_1|+1 = l-n-1 & &&
    &|t_i| = |t^x_i| = |t^y_i| = |t^z_i| = |t^w_i| = 1 \\
    &|x_i| = |y_i| = |z_i| = |w_i| = 0
  \end{align*}
  For infinitely many choices of $k,l,m$, and $n$, the gradings in each row and column will
  be distinct.

  The differential has the following form:

  \begin{align*}
    \df a_1 &= 0 & \df t_1 &= 1+ x_1 (1+a_1a_2+b_1b_2) \\
    \df a_2 &= y_1c_0c_1b_2 & \df t_2 &= 1+ (1+a_2a_1)y_1(1+c_0f +
    c_0c_1e_0) +a_2b_1e_1 \\
    && \df t_3 &= 1+ [1+b_2b_1 + (1 + b_2b_1) dc_1]w_1 \\
    \df b_1 &= a_1y_1c_0c_1 & \df t_4 &= 1+ (1+f c_0) z_1 + c_1d \\
    \df b_2 &= 0 & \df t_0 &= 1 + x_{k+1}  y_{l+1}  z_{m+1}  w_{n+1} \\
    && \\
    \df c_0 &= \df c_1 =  0 & \df t^x_i &= 1+ x_ix_{i+1} \text{ and similarly for  } t^y_i, t^z_i, t^w_i \\
    && \\
    \df d &= e_1c_0  &     \df e_0 &= (1+b_2b_1) e_1  \\
    \df f &= c_1e_1  &     \df e_1 &= 0  
  \end{align*}

  Since we assume that only the $x_i$, $y_i$, $z_i$, and $w_i$ have
  zero grading, there is a unique augmentation that sends these
  generators to $1$ and all others to $0$.  Abusing notation to
  identify generators and their duals, we see that the linearized
  cohomology is given by:
  \begin{equation}
    LCH^k_\aug(K) = \langle [a_i], [b_i], [c_i], [d], [f], [t] \rangle,
  \end{equation}
  where $[t]$ is once again any one of the right cusps.  

  Duality pairs the $[a_i]$, the $[b_i]$, $[c_0]$ with $[f]$, and
  $[c_1]$ with $[d]$.  There are no other nontrivial cup products; in
  fact, all cup products between cocycles (beyond those involved in
  the duality pairing) vanish at the cochain level.  Thus, it follows
  that the $m_3$ products between triples of cocycles yield two Massey
  products:
  \begin{align*}
    \{[c_0],[c_1],[b_2]\} &= [a_2], \\
    \{[a_1],[c_0],[c_1]\} &= [b_1].
  \end{align*}
  Since the only class in the image of the cup product is $[t]$, the
  Massey products above lie in $\overline{LCH}^*_\aug(K)$, and hence
  are nontrivial.  Under the assumption that the generators $a_i$,
  $b_i$, and $c_i$ have distinct gradings, we can then easily see that
  there are no nontrivial Massey products in these gradings in the
  linearized cohomology of the Legendrian mirror.  Hence, the knot $K$
  and its Legendrian mirror are not Legendrian isotopic even though
  their linearized cohomology rings are isomorphic.
\end{proof}

\subsection{Higher-Order Massey Products}
\label{ssec:hop}

As in the previous subsections, we can show that the higher-order
Massey products are also nontrivial.

\begin{proof}[Completion of the Proof of Theorem~\ref{thm:mir-ex}]
  The Legendrian knot $K_n$ in Figure~\ref{fig:high-massey-ex} is not
  isotopic to its Legendrian mirror. The two knots can be
  distinguished using the $(n+1)^{\text{st}}$-order Massey products on
  the linearized contact cohomology but not by their linearized
  contact cohomology rings or their $m^{\text{th}}$-order Massey
  products for $m \leq n$.

  By a similar calculation to the previous examples, one can show that
  the cup products (besides those associated with duality) and the
  lower-order Massey products all vanish, so the
  $(n+1)^{\text{st}}$-order Massey product lies in
  $\overline{LCH}^*_\aug(K)$. Further, the cup product and lower-order
  Massey products vanish at the chain level, so we have the following
  two Massey products whose gradings are nonsymmetric:
  \begin{align*}
    \{[c_0], \ldots, [c_n], [b_2] \} &= [a_2] \\
    \{[a_1], [c_0], \ldots, [c_n] \} &= [b_1].
  \end{align*}
  It follows that the knot $K_n$ is not isotopic to its Legendrian
  mirror.
\end{proof}

\begin{figure}[ht]
  \relabelbox \small {%
    \centerline{\epsfxsize=5.5in \epsfbox{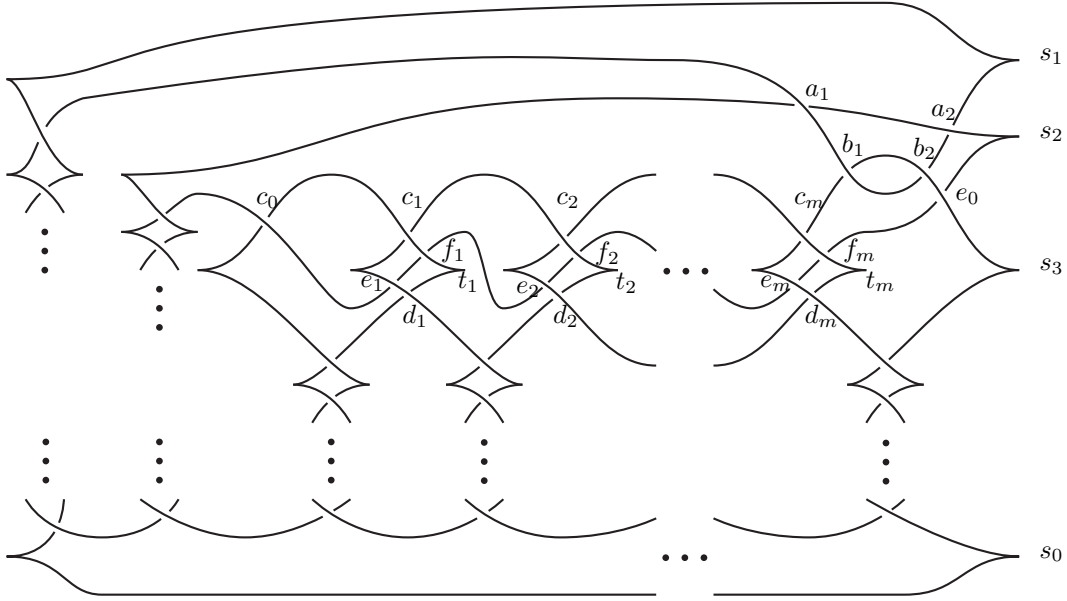}}}
  \relabel{1}{$a_1$}
  \relabel{2}{$a_2$}
  \relabel{3}{$b_1$}
  \relabel{4}{$b_2$}
  \relabel{5}{$c_0$}
  \relabel{6}{$c_1$}
  \relabel{7}{$c_2$}
  \relabel{8}{$c_m$}
  \relabel{9}{$d_1$}
  \relabel{a}{$d_2$}
  \relabel{b}{$d_m$}
  \relabel{c}{$e_1$}
  \relabel{d}{$e_2$}
  \relabel{e}{$e_m$}
  \relabel{f}{$e_0$}
  \relabel{g}{$f_1$}
  \relabel{h}{$f_2$}
  \relabel{i}{$f_m$}
  \relabel{j}{$t_1$}
  \relabel{k}{$t_2$}
  \relabel{l}{$t_m$}
  \relabel{m}{$s_1$}
  \relabel{n}{$s_2$}
  \relabel{o}{$s_3$}
  \relabel{p}{$s_0$}
  \endrelabelbox
  \caption{This knot is distinguished from its Legendrian mirror by
    its order $n+1$ Massey products.}
  \label{fig:high-massey-ex}
\end{figure}

\begin{rem}
  Using the ``splashes'' of \cite{fuchs:augmentations} or the ``dips''
  of \cite{rulings}, it is possible to show that the $A_\infty$
  structure on the linearized cochain complex is $A_\infty$
  quasi-isomorphic to one for which $m_k = 0$ for all $k \geq 4$.  As
  the examples above show, however, this does not mean that the
  $A_\infty$ structure $\bmu$ on the linearized contact
  cohomology is trivial for $k \geq 4$.
\end{rem}

\section{Products and Higher Order Linearizations}
\label{sec:higher-order}

In this section, we explore the relationship between the $A_\infty$
structure on the linearized contact cohomology, associated product
structures, and Chekonov's order $n$ linearizations.  



\subsection{The Minimal Model Theorem, Revisited}
\label{ssec:min-model}

We begin by sketching the proof of the Minimal Model
Theorem~\ref{thm:kad} following Markl's formulae in
\cite{markl:transfer}; see also \cite{kajiura:open-string,
  kontsevich-soibelman, merkulov, smirnov:book}.  First, let us
describe the construction of the maps of $\bmu$. The fact that we are
working over the field $\zz_2$ allows us to choose maps $i: H^*(V) \to
V$, $p: V \to H^*(V)$, and $h: V \to V$ such that:
\begin{equation}
  p \circ i = \id \quad \text{and} \quad \id + i \circ p 
  = \delta h + h \delta.
\end{equation}
We next consider the set $\Gamma_k$ of rooted planar trees with $k$
leaves (the root edge is not counted among the $k$ leaves) and at
least trivalent internal vertices.  For each $T \in \Gamma_k$, we
construct a map $g_T: V^{\otimes k} \to V$ by placing the inputs in
order along the $k$ leaves, an $m_k$ at each $(k+1)$-valent internal
vertex, and an $h$ at each internal edge; see
Figure~\ref{fig:rooted-trees}.  The map $g_T$ is then defined by
appropriately inserting arguments and composing maps from the leaves
down to the root.  We then define $g_1 = \delta$ and for $k \geq 2$,
the maps:
\[
g_k = \sum_{T \in \Gamma_k} g_T.
\]
These maps form a sequence $\mathbf{g}$.

\begin{figure}[ht]
  \relabelbox \small {%
    \centerline{\epsfbox{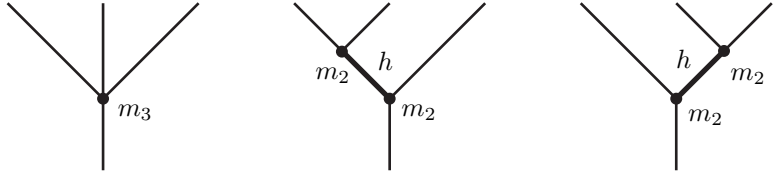}}}
  \relabel{1}{$m_3$}
  \relabel{2}{$m_2$}
  \relabel{3}{$m_2$}
  \relabel{4}{$m_2$}
  \relabel{5}{$m_2$}
  \relabel{6}{$h$}
  \relabel{7}{$h$}
  \endrelabelbox
  \caption{The rooted trees that make up the map $g_3$.  The rightmost
  tree gives the map $g_T(a,b,c) = m_2(a,h \circ m_2(b,c))$.}
  \label{fig:rooted-trees}
\end{figure}

The products $\bmu$ are then defined by:
\[
\mu_k = p \circ g_k \circ (i \otimes \cdots \otimes i).
\]
The product $\mu_3: H^*(V)^{\otimes 3} \to H^*(V)$, for example, is
defined as follows, where we write $i(\alpha_k) = a_k$:
\[\begin{aligned}
  \mu_3(\alpha_1, \alpha_2, \alpha_3) = & \,\, p \bigl( m_3( a_1, a_2,
  a_3) \\ &+ m_2\left(a_1, h \circ m_2(a_2,a_3)\right) + m_2\left(h
    \circ m_2(a_1,a_2), a_3 \right) \bigr).
\end{aligned}\] 

The maps $i$, $p$, and $h$ can also be extended to sequences of maps
$\mathbf{i}$, $\mathbf{p}$, and $\mathbf{h}$.  The map $i_k$, for
example, is defined by $i_k = h \circ g_k \circ (i \otimes \cdots
\otimes i)$.  The formulae for $p_k$ and $h_k$ are also based on
rooted planar trees, but are somewhat more involved.

\begin{prop}[Markl \cite{markl:transfer}] 
  \label{prop:a-infty-fmlae}
  The maps $\bmu$ give an $A_\infty$ structure on $H^*(V)$, the maps
  $\mathbf{i}$ and $\mathbf{p}$ are $A_\infty$ morphisms, and the maps
  $\mathbf{h}$ are an $A_\infty$ homotopy between $\mathbf{i} \circ
  \mathbf{p}$ and the identity on $V$.
\end{prop}

Here, an $A_\infty$ homotopy between two $A_\infty$ morphisms $f,g:
(V, \m) \to (W, \n)$ is a sequence of degree $-1$ maps $h_n:
V^{\otimes n} \to V$ that satisfy:
\[
f_n + g_n = \sum_{i+j+k=n} h_{i+1+k}\circ(1^{\otimes i}\otimes
m_j\otimes 1^{\otimes k}) +\sum_{\substack{1\leq k \leq r\leq n \\
    i_1+\ldots+i_r=n}} n_r\circ (f_{i_1} \otimes \cdots \otimes
f_{i_{k-1}} \otimes h_{i_k} \otimes g_{i_{k+1}} \otimes \cdots \otimes
g_{i_r} ).
\]

\begin{rem} \label{rem:a-n} In particular, we have that $\mathbf{i}$
  is the $A_\infty$ quasi-isomorphism promised by the Minimal Model
  Theorem.  Note that the proposition also yields $A_n$ morphisms and
  homotopies by stopping the construction at any finite step.
\end{rem}

\subsection{$A_\infty$ Structures Determine Massey Products}
\label{ssec:a-infty-massey}

The relationship between the $A_\infty$ structure on $H^*(V)$ and the
Massey products is straightforward to state:

\begin{prop}[Kadeishvili \cite{kadeishvili}] \label{prop:a-infty-massey}
  Given $\alpha_k \in H^*(V)$, $k=1,2,3$, such that \[\mu_2(\alpha_1,
  \alpha_2) = 0 = \mu_2(\alpha_2, \alpha_3),\] the projection of
  $\mu_3(\alpha_1, \alpha_2, \alpha_3)$ to $\tilde{H}^*(V)$ agrees
  with the Massey product $\{\alpha_1, \alpha_2, \alpha_3\}$.
\end{prop}

To see this, choose $x = h \circ m_2(a_1,a_2)$ and $y = h \circ
m_2(a_2,a_3)$.  Notice that:
\begin{align*}
  \delta x &= m_2(a_1,a_2) + i \circ p \circ m_2(a_1,a_2) + h \delta
  m_2(a_1,a_2) \\
  &= m_2(a_1,a_2).
\end{align*}
Note that the last term in the first line vanishes since
$m_2(a_1,a_2)$ is a cycle, and the second-to-last term vanishes since
$m_2(a_1,a_2)$ represents the zero cohomology class by assumption.  A
similar fact holds for $y$, so we may take $x$ and $y$ to be the
elements required for the definition of the Massey product
$\{\alpha_1, \alpha_2, \alpha_3\}$.  Now we need only compute that:
\begin{align*}
  \mu_3(\alpha_1, \alpha_2, \alpha_3) & = p \bigl( m_3( a_1, a_2, a_3)
  + m_2\left(a_1, h \circ m_2(a_2,a_3)\right) \\ & \quad + m_2\left(h
    \circ m_2(a_1,a_2), a_3 \right) \bigr) \\
  &= p ( m_3(a_1,a_2,a_3) + m_2(a_1,y) + m_2(x,a_3)),
\end{align*}
which, by definition, projects to the Massey product.

In fact, this is the base case for a proof of a similar statement for
order $n$ Massey products defined using the full $A_n$ structure.
The proof of this folk theorem is a straightforward generalization of
that in \cite{lpwz} using the language introduced above.

\subsection{$A_\infty$ Structure on $LCH_\epsilon^*$ and Higher Order
  Linearizations.}

We are finally ready to prove Theorem~\ref{thm:a-infty-high-order},
which states that the $A_n$ structure on $LCH^*_\aug$ is strictly
stronger than the $n^{\text{th}}$-order linearized cohomology.  Before
proving the theorem, however, we need to introduce one more algebraic
object, Stasheff's \dfn{tilde construction} $(\tilde{B}^{n}(V), d^n)$
of an $A_n$ algebra $(V, \m)$ \cite{stasheff:h-space-2}.\footnote{For
  readers familiar with the bar construction, please note that the
  tilde construction is truncated after the order $n$ tensors.}  The
chains of this complex lie in
\[\tilde{B}^n(V) = \bigoplus_{k=1}^n V^{\otimes k},\] while 
the differential $d^n$ is defined componentwise by:
\begin{equation} \label{eqn:bar-diff} d^n|_{V^{\otimes a}} =
  \sum_{i+j+k=a} 1^{\otimes i} \otimes m_j \otimes 1^{\otimes k}.
\end{equation}
That this differential satisfies $(d^n)^2 = 0$ follows from the
defining $A_\infty$ equation \eqref{main-a-infty}.  The reason that we
introduce the tilde construction is the following result. 

\begin{lem} \label{lem:bar} The $n^{\text{th}}$-order linearized
  cochain complex with respect to $\aug$ is the tilde construction of
  $(A^*, \m^\aug)$.
\end{lem}

\begin{proof}
  Clearly, we have that $A_{(n)} = \bigoplus_{i = 1}^n A^{\otimes i}$
  and the differential $\dfe_{(n)}$ is equivalent to the following,
  essentially because of the Leibniz rule and the fact that any term
  of length greater than $n$ becomes zero in $A_{(n)}$:
  \[
  \dfe_{(n)}| = \sum_{i+j+k \leq n } 
  1^{\otimes i} \otimes
  \dfe_j \otimes 1^{\otimes k}.
  \]
  Dually, it is now easy to see that the $n^{\text{th}}$-order
  linearized cochain complex has cochains in $\bigoplus_{i = 1}^n
  (A^*)^{\otimes i}$ and codifferential $\dde_{(n)}$ defined precisely
  as in Equation~\eqref{eqn:bar-diff}.
\end{proof}

It is straightforward to see that $A_n$ morphisms and homotopies
translate to similar notions for the tilde construction (see, for
example, \cite{markl:transfer, smirnov:book}).

\begin{lem} \label{lem:bar-morphism}
  \begin{enumerate}
  \item An $A_n$ morphism $f: (V, \m) \to (W, \n)$ determines a chain
    map $\tilde{B}^nf: \tilde{B}^nV \to \tilde{B}^nW$ whose
    $V^{\otimes a}$ component is
    \[ \sum_r \sum_{i_1 + \cdots + i_r = a} f_{i_1} \otimes \cdots
    \otimes f_{i_r}.\]
  \item An $A_n$ homotopy $h: (V, \m) \to (W, \n)$ between $f$ and $g$
    determines a chain homotopy $\tilde{B}^nh: \tilde{B}^nV \to
    \tilde{B}^nW$ between $\tilde{B}^nf$ and $\tilde{B}^ng$.
  \end{enumerate}
\end{lem}

We are now ready to prove Theorem~\ref{thm:a-infty-high-order}.

\begin{proof}[Proof of Theorem~\ref{thm:a-infty-high-order}]
  We begin by proving that the $A_n$ structure on the linearized
  cohomology determines the $n^{\text{th}}$-order linearized
  cohomology.  Fix an augmentation $\aug$. By the remarks in
  Section~\ref{ssec:min-model}, we know that the $A_n$ structures on
  the linearized cochain complex $(A^*, \delta^\aug)$ and the
  linearized cohomology $LCH^*_\aug$ are $A_\infty$ homotopy
  equivalent. The lemma above then implies that their tilde
  constructions are chain homotopy equivalent, and hence have
  isomorphic cohomologies.  Since the $A_n$ structure on $LCH^*_\aug$
  determines the cohomology of its tilde construction, it also
  determines the cohomology of the tilde construction of the $A_n$
  structure on $A^*$ which, by Lemma~\ref{lem:bar}, is simply
  $LCH^*_\aug(K,n)$.  This proves the first half of the theorem.

  To prove that the $A_n$ structure is strictly stronger, we observe
  that since order $n$ Massey products be used to distinguish the
  Legendrian knots in Theorem~\ref{thm:mir-ex} from their Legendrian
  mirrors, Proposition~\ref{prop:a-infty-massey} and its order $n$
  generalization implies that the $A_n$ structures also distinguish
  these knots.  On the other hand, the higher-order cohomologies can
  never distinguish a Legendrian knot $K$ from its Legendrian mirror
  $\overline{K}$.  To see why, notice that the reflection map $\tau:
  A_{(n)} \to A_{(n)}$ defined by:
  \[
  \tau(a_1 \otimes \cdots \otimes a_n) = a_n \otimes \cdots \otimes
  a_1
  \]
  gives a quasi-isomorphism (but not necessarily a tame isomorphism)
  between $(A_{(n)}, \df_{(n)})$ and $(\overline{A}_{(n)},
  \overline{\df}_{(n)})$.
\end{proof}

\subsection{An Alternative Proof when $n=2$}

In this final section, we present a more down-to-earth proof that, for
a fixed augmentation, the linearized cohomology ring is strictly
stronger than the order $2$ linearized cohomology, i.e.\ the $n=2$
case of Theorem~\ref{thm:a-infty-high-order}.

First, write $A_{(2)}$ as $A \oplus A^{\otimes 2}$. For ease of
exposition, we drop the augmentation $\epsilon$ from the notation. In
this notation, the codifferential can be recorded by:
\begin{equation}
  \dd_{(2)} = \begin{bmatrix}
    \dd & m_2 \\ 0 & \dd^\otimes
  \end{bmatrix},
\end{equation}
which implies the following lemma:

\begin{lem} \label{lem:cup-mapping-cone} The second order linearized
  cochain complex is the mapping cone of the degree $1$ chain map
  $m_2: ((A^*)^{\otimes 2}, \dd^\otimes) \to (A^*, \dd)$.
\end{lem}

Associated to this mapping cone we have the standard long exact
sequence:
\begin{equation*}
  \cdots \to  LCH^k(K) \to 
  LCH^k(K,2) \to H^k ((A^*)^{\otimes 2}, \dd^\otimes) \stackrel{d^*}{\to}
  LCH^{k+1}(K) \to \cdots.
\end{equation*}
Since we are working over a field, the K\"unneth formula gives us:
\[
H^{k} ((A^*)^{\otimes 2}, \dd^\otimes)=\bigoplus_{i+j=k} LCH^i\otimes LCH^j.
\]
Moreover, we know that the connecting homomorphism $d^*$ is induced by
$m_2.$ That is, it is given by the cup product $\mu_2$. Again, since
we are working over a field, the short exact sequences into which the
long exact sequence above decomposes must all split. This gives:
\begin{equation}
  LCH^k(K,2) \simeq \ker \mu_2 \oplus 
  \bigl(LCH^k(K) / \img \mu_2 \bigr).
\end{equation}
In other words, the second-order linearization is determined by the
image and kernel of the cup product map on linearized contact
cohomology.  Thus, the linearized cohomology and the cup product
determine the second-order linearized cohomology.

\def\cprime{$'$} \def\polhk#1{\setbox0=\hbox{#1}{\ooalign{\hidewidth
  \lower1.5ex\hbox{`}\hidewidth\crcr\unhbox0}}}
\providecommand{\bysame}{\leavevmode\hbox to3em{\hrulefill}\thinspace}
\providecommand{\MR}{\relax\ifhmode\unskip\space\fi MR }
\providecommand{\MRhref}[2]{%
  \href{http://www.ams.org/mathscinet-getitem?mr=#1}{#2}
}
\providecommand{\href}[2]{#2}


\begin{thebibliography}{10}

\bibitem{bennequin}
D.~Bennequin, \emph{Entrelacements et equations de {P}faff}, Asterisque
  \textbf{107--108} (1983), 87--161.

\bibitem{chv}
Yu. Chekanov, \emph{Differential algebra of {L}egendrian links}, Invent. Math.
  \textbf{150} (2002), 441--483.

\bibitem{chv-pushkar}
Yu. Chekanov and P.~Pushkar, \emph{Combinatorics of {L}egendrian links and the
  {A}rnol'd $4$-conjectures}, Russ. Math. Surv. \textbf{60} (2005), no.~1,
  95--149.

\bibitem{ding-geiges:knots}
F.~Ding and H.~Geiges, \emph{Legendrian knots and links classified by classical
  invariants}, Commun. Contemp. Math. \textbf{9} (2007), no.~2, 135--162.

\bibitem{high-d-duality}
T.~Ekholm, J.~Etnyre, and J.~Sabloff, \emph{A duality exact sequence for
  {L}egendrian contact homology}, Preprint available as arXiv:0803.2455, 2008.

\bibitem{yasha:icm}
Ya. Eliashberg, \emph{Invariants in contact topology}, Proceedings of the
  International Congress of Mathematicians, Vol. II (Berlin, 1998), no. Extra
  Vol. II, 1998, pp.~327--338 (electronic).

\bibitem{yasha-fraser}
Ya. Eliashberg and M.~Fraser, \emph{Classification of topologically trivial
  {L}egendrian knots}, Geometry, topology, and dynamics (Montreal, PQ, 1995),
  Amer. Math. Soc., Providence, RI, 1998, pp.~17--51.

\bibitem{etnyre:knot-intro}
J.~Etnyre, \emph{Legendrian and transversal knots}, Handbook of knot theory,
  Elsevier B. V., Amsterdam, 2005, pp.~105--185.

\bibitem{etnyre-honda:knots}
J.~Etnyre and K.~Honda, \emph{Knots and contact geometry {I}: Torus knots and
  the figure eight knot}, J. Symplectic Geom. \textbf{1} (2001), no.~1,
  63--120.

\bibitem{ens}
J.~Etnyre, L.~Ng, and J.~Sabloff, \emph{Invariants of {L}egendrian knots and
  coherent orientations}, J. Symplectic Geom. \textbf{1} (2002), no.~2,
  321--367.

\bibitem{fuchs:augmentations}
D.~Fuchs, \emph{Chekanov-{E}liashberg invariant of {L}egendrian knots:
  existence of augmentations}, J. Geom. Phys. \textbf{47} (2003), no.~1,
  43--65.

\bibitem{fuchs-ishk}
D.~Fuchs and T.~Ishkhanov, \emph{Invariants of {L}egendrian knots and
  decompositions of front diagrams}, Mosc. Math. J. \textbf{4} (2004), no.~3,
  707--717.

\bibitem{f-t}
D.~Fuchs and S.~Tabachnikov, \emph{Invariants of {L}egendrian and transverse
  knots in the standard contact space}, Topology \textbf{36} (1997),
  1025--1053.

\bibitem{kadeishvili}
T.~Kadei{\v{s}}vili, \emph{On the theory of homology of fiber spaces}, Uspekhi
  Mat. Nauk \textbf{35} (1980), no.~3(213), 183--188, International Topology
  Conference (Moscow State Univ., Moscow, 1979).

\bibitem{kajiura:open-string}
H.~Kajiura, \emph{Noncommutative homotopy algebras associated with open
  strings}, Rev. Math. Phys. \textbf{19} (2007), no.~1, 1--99.

\bibitem{kontsevich-soibelman}
M.~Kontsevich and Y.~Soibelman, \emph{Homological mirror symmetry and torus
  fibrations}, Symplectic geometry and mirror symmetry ({S}eoul, 2000), World
  Sci. Publ., River Edge, NJ, 2001, pp.~203--263.

\bibitem{lpwz}
D.-M. Lu, J.~Palmieri, Q.-S. Wu, and J.~Zhang, \emph{{A}-infinity structure on
  {E}xt-algebras}, Preprint, 2007.

\bibitem{markl:transfer}
M.~Markl, \emph{Transferring {$A\sb \infty$} (strongly homotopy associative)
  structures}, Rend. Circ. Mat. Palermo (2) Suppl. (2006), no.~79, 139--151.

\bibitem{melvin-shrestha}
P.~Melvin and S.~Shrestha, \emph{The nonuniqueness of {C}hekanov polynomials of
  {L}egendrian knots}, Geom. Topol. \textbf{9} (2005), 1221--1252.

\bibitem{merkulov}
S.~A. Merkulov, \emph{Strong homotopy algebras of a {K}\"ahler manifold},
  Internat. Math. Res. Notices (1999), no.~3, 153--164.

\bibitem{lenny:computable}
L.~Ng, \emph{Computable {L}egendrian invariants}, Topology \textbf{42} (2003),
  no.~1, 55--82.

\bibitem{ns:augm-rulings}
L.~Ng and J.~Sabloff, \emph{The correspondence between augmentations and
  rulings for {L}egendrian knots}, Pacific J. Math. \textbf{224} (2006), no.~1,
  141--150.

\bibitem{ost:leg-trans}
P.~Ozsv{\'a}th, Z.~Szab{\'o}, and D.~Thurston, \emph{Legendrian knots,
  transverse knots and combinatorial {F}loer homology}, Geom. Topol.
  \textbf{12} (2008), no.~2, 941--980.

\bibitem{rutherford:kauffman}
D.~Rutherford, \emph{Thurston-{B}ennequin number, {K}auffman polynomial, and
  ruling invariants of a {L}egendrian link: the {F}uchs conjecture and beyond},
  Int. Math. Res. Not. (2006), Art. ID 78591, 15.

\bibitem{lecnotes}
J.~Sabloff, \emph{Invariants for {L}egendrian knots from contact homology}, In
  preparation.

\bibitem{rulings}
\bysame, \emph{Augmentations and rulings of {L}egendrian knots}, Int. Math.
  Res. Not. (2005), no.~19, 1157--1180.

\bibitem{duality}
\bysame, \emph{Duality for {L}egendrian contact homology}, Geom. Topol.
  \textbf{10} (2006), 2351--2381 (electronic).

\bibitem{smirnov:book}
V.~Smirnov, \emph{Simplicial and operad methods in algebraic topology},
  Translations of Mathematical Monographs, vol. 198, American Mathematical
  Society, Providence, RI, 2001, Translated from the Russian manuscript by G.
  L. Rybnikov.

\bibitem{stasheff:h-space-2}
J.~Stasheff, \emph{Homotopy associativity of $h$-spaces. ii}, Trans. Amer.
  Math. Soc. \textbf{108} (1963), no.~2, 293--312.

\end{thebibliography}

\end{document}